\documentclass[11pt]{amsart}

\usepackage{latexsym}
\usepackage{amsmath}
\usepackage{amsfonts}
\usepackage{amssymb}
\usepackage{enumitem}
\usepackage {tikz}
\usepackage{color}
\usetikzlibrary {positioning}
\usetikzlibrary{arrows,shapes}
\usepackage{color}
\newtheorem{theorem}{Theorem}[section]
\newtheorem{lemma}[theorem]{Lemma}
\newtheorem{corollary}[theorem]{Corollary}
\newtheorem{proposition}[theorem]{Proposition}

\newtheorem{sublemma}{}[theorem]

\theoremstyle{definition}

\theoremstyle{remark}

\numberwithin{equation}{section}
\usepackage{stackengine}

\stackMath

\newcommand{\ba}{\backslash}

\begin{document}

\title[Constructing Binary Matroids]{Complementation, Local Complementation, and Switching in Binary Matroids}

\author{James Oxley}
\address{Mathematics Department\\
Louisiana State University\\
Baton Rouge, Louisiana}
\email{oxley@math.lsu.edu}

\author{Jagdeep Singh}
\address{Mathematics Department\\
Louisiana State University\\
Baton Rouge, Louisiana}
\email{jsing29@math.lsu.edu}

\subjclass{05B35, 05C25}
\keywords{complementation, local complementation, switching, binary matroid}
\date{\today}
\thanks{This paper is dedicated to Joseph Kung who has done so much for our subject.}

\begin{abstract}
In 2004, Ehrenfeucht, Harju, and Rozenberg showed that any graph on a vertex set $V$ can be obtained from a complete graph on $V$ via a sequence of the operations of complementation, switching edges and non-edges at a vertex, and local complementation. The last operation involves taking the complement in the neighbourhood of a vertex.  In this paper, we consider natural generalizations of these operations for binary matroids and explore their behaviour. We characterize all binary matroids obtainable from the binary projective geometry of rank $r$ under the operations of complementation and switching. Moreover, we show that not all binary matroids of rank at most $r$ can be obtained from a projective geometry of rank $r$ via a sequence of the three generalized operations. We introduce a fourth operation and show that, with this additional operation, we are able to obtain all binary matroids.
\end{abstract}

\maketitle

\section{Introduction}
\label{intro}
We only consider simple graphs and simple binary matroids in this paper. Our notation and terminology will follow \cite{ox1}.
Ehrenfeucht et al. consider three natural operations on a graph $G=(V,E)$ in \cite{aer04}. Let $K_V$ be the complete graph on the vertex set $V$. The \textbf{complement} $\omega(G)$ of $G$ is $(V, E(K_V)-E)$. Let $x$ be a vertex of $G$ and $E_x(G)$ be the set of edges of $G$ meeting $x$. The graph $\sigma_x(G)$, the \textbf{switching} of $G$ at $x$, is $(V, E \bigtriangleup E_x(K_V))$. Thus, at $x$, we interchange edges and non-edges. Let $N_G(x)$ be the set of neighbours of $x$. The \textbf{local complementation} $\lambda_x(G)$ of $G$ at $x$ is $(V, E \bigtriangleup E(K_{N_G(x)}))$, that is, in the neighbourhood of $x$, we interchange edges and non-edges. 

Ehrenfeucht et al.\cite{aer04} showed that complementation can be obtained by a sequence of operations of switching and local complementation. Their main result is the following.

\begin{theorem}
\label{maingraph}
Every graph on the vertex set $V$ can be obtained from $K_V$ via a sequence of switchings and local complementations.
\end{theorem}

In this paper, we try to generalize this theorem to binary matroids. Just as every $n$-vertex simple graph is a subgraph of $K_n$, every simple binary matroid of rank at most $r$ is a restriction of $PG(r-1,2)$. Moreover, the three graph operations defined above have natural analogues for binary matroids.

Throughout this paper, we denote the rank-$r$ binary projective geometry ${PG(r-1,2)}$ by $P_r$. We call cocircuits and hyperplanes of $P_r$ \textbf{projective cocircuits} and \textbf{projective hyperplanes}, respectively. For a given binary matroid $M$, we fix a binary projective geometry $P_r$ of which $M$ is a restriction.  The \textbf{complement} $\omega(M)$ of $M$ in $P_r$ is the matroid on the ground set $E(M) \bigtriangleup E(P_r)$, that is, we change the elements of $P_r$ present in the ground set of $M$ to non-elements and vice versa. Observe that the switching of the graph $G$ with respect to a vertex $x$ is obtained by complementing inside the vertex bond of $K_V$ at $x$. Since bonds in graphs correspond to cocircuits in matroids, a natural binary-matroid generalization of the switching operation in graphs is to complement inside a cocircuit of $P_r$. The \textbf{switching}  $\sigma_{C^*}(M)$ of $M$ in $P_r$ with respect to a cocircuit $C^*$ of $P_r$ is the matroid on the ground set $E(M) \bigtriangleup C^*$, that is, we change the elements of $C^*$ present in the ground set of $M$ to non-elements and vice versa. The \textbf{local complementation} $\lambda_{C^*}(M)$ of $M$ in $P_r$ with respect to a projective cocircuit $C^*$ is the matroid on the ground set $E(M) \bigtriangleup ({\rm{cl}}_{P_r}(C^* \cap E(M)) - C^*)$, that is, we complement inside ${\rm{cl}}_{P_r}(C^* \cap E(M)) - C^*$, where ${\rm{cl}}_{P_r}$ denotes projective closure, that is, closure in $P_r$. For example, $\lambda_{C^*}(P_r)$ is the rank-$r$ binary affine geometry, $AG(r-1,2)$, which we will write as $A_r$ here.

In Section 3, we shall observe that the operations of complementation and switching commute with each other and that a composition of switchings is a switching. We use these observations to characterize the matroids in the same orbit as $P_r$ under the action of the operations of switching and complementation. In Section 4,  we show that the switchings with respect to row-cocircuits of  some fixed standard representation of $P_r$ generate all switchings.

In Section 5, we show that complementation can be written in terms of switchings and local complementations and that not all binary matroids of rank at most $r$ can be obtained from $P_r$ via a sequence of the operations of complementation, switching, and local complementation.

In Section 6, we introduce the pointed-swap operation and show that this new operation along with the operations of switching and complementation are enough to transform $P_r$ into any binary matroid of rank at most $r$. Moreover, we show that any non-empty binary matroid other than $U_{1,1}$ can be obtained from $P_r$ via the operations of local complementation and pointed swaps.

\section{Preliminaries}
\label{prelim}

Let $\mathcal{M}_r$ denote the set of all binary matroids that are restrictions of a fixed copy of $P_r$ and let Sym$(\mathcal{M}_r)$ denote the symmetric group {\color{black} on} $\mathcal{M}_r$.  For any subset $A$ of $E(P_r)$ and a matroid $M \in \mathcal{M}_r$, let $M \triangle A$ denote the matroid on the ground set $E(M) \bigtriangleup A$, the symmetric difference of $E(M)$ and $A$. With the above notation, our three operations with respect to $P_r$ can be written in the following way:

\begin{enumerate}
    \item {\bf Complementation:} $\omega(M)=M \triangle E(P_r)$.
    \item {\bf Switching:} $\sigma_{C^*}(M) = M \triangle C^* $.
    \item {\bf Local Complementation:} $\lambda_{C^*}(M) = M \triangle           [{\rm{cl}}_{P_r}(E(M) \cap C^*) - C^*]$.
\end{enumerate}

From the above notation, it is clear that $\omega^2, \sigma_{C^*}^2,$ and $\lambda_{C^*}^2$ are all identity operations, that is, $\omega^2(M)=\sigma_{C^*}^2(M)= \lambda_{C^*}^2(M)=M$. {\color{black} Note that the operations of switching, complementation, and local complementation generate a subgroup  of Sym$(\mathcal{M}_r)$.} The following lemma is an immediate consequence of the fact that the operation of symmetric difference is commutative. It implies that the operations of complementation and switching commute.

\begin{lemma}
\label{symdiff}
For a binary matroid $M$ of rank at most $r$ and $X, Y \subseteq E(P_r)$, 
$$M \triangle (X \bigtriangleup Y) = (M \triangle Y) \triangle X.$$
\end{lemma}

The following is a straightforward consequence of the last lemma. 

\begin{lemma}
\label{hypcomp}
Let $M$ be a restriction of $P_r$ and $C^*$ be a projective cocircuit. Then $\omega \sigma_{C^*}(M) = M \triangle (E(P_r)-C^*)$.

\end{lemma}

We shall refer to the operation $M \triangle (E(P_r)-C^*)$ as \textbf{complementation inside the projective hyperplane} $E(P_r)-C^*$. It may appear that the matroid generalization of the switching operation is stronger than the switching operation for graphs since, in graphs, we do complementation inside a vertex bond of the complete graph while, in matroids we do complementation inside any projective cocircuit. The following well-known result shows that this is not the case.

\begin{lemma}
\label{switch1}
A set $B$ of edges of a complete graph $K_V$ on $V$ is a bond if and only if it can be written as {\color{black} the  symmetric difference of some set of} vertex bonds of $K_V$.
\end{lemma}

This implies that we can do complementation inside any bond of the complete graph via a sequence of switchings.

\section{Switching and Complementation}

In this section, we characterize all the matroids obtainable from $P_r$ using the operations of complementation and switching.

The following result is known for graphs \cite{harju1, hage1, seid}.

\begin{theorem}
The graphs obtainable from a complete graph on $n$ vertices using the operations of complementation and switching consist of all complete bipartite graphs on $n$ vertices together with their complements.
\end{theorem}

For binary matroids, we start with the following elementary result.

\begin{lemma}
\label{symm2}
Let $C_1^*$ and $C_2^*$ be two distinct cocircuits of $P_r$. Then $C_1^* \bigtriangleup C_2^*$ is also a cocircuit of $P_r$.
\end{lemma}

\begin{proof}
Note that $C_1^* \bigtriangleup C_2^*$ is a disjoint union of cocircuits of $P_r$. Since $|C_1^* \bigtriangleup C_2^*| = 2^{r-1}$ and the number of elements in any projective cocircuit is $2^{r-1}$, we deduce that $C_1^* \bigtriangleup C_2^*$ is a cocircuit of $P_r$. 
\end{proof}


Recall that we are denoting the operation of complementation by $\omega$. We use $\iota$ to denote the identity operator.

\begin{lemma}
\label{switch and complemetation 0}

 Every product of a sequence of the operations of complementation and switching is equal to $\iota$, to $\omega$, to $\sigma_{C^*}$ for some cocircuit $C^*$ of $P_r$, or to $\omega \sigma_{C^*}$.

\end{lemma}

\begin{proof}
Observe that $\sigma_{C_1^*}\sigma_{C_2^*} \dots \sigma_{C_k^*}=\sigma_{C_1^* \bigtriangleup C_2^* \bigtriangleup \dots \bigtriangleup C_k^*}$. Since the operations of complementation and switching commute, and both have order two, the result follows.

\end{proof}

This lemma immediately implies the following result.

\begin{proposition}
\label{switch and complement}

The binary matroids obtainable from $P_r$ using the operations of complementation and switching are all of the matroids that are isomorphic to one of $P_r$, $U_{0,0}$, $P_{r-1}$, and $A_{r}$.
\end{proposition}

By combining this result with Lemma \ref{hypcomp}, we obtain the following.

\begin{corollary}
\label{switch and complement gen}
 Two distinct binary matroids $M_1$ and $M_2$ are in the same orbit under the action of the group generated by the operations of switching and complementation on $\mathcal{M}_r$ if and only if one of the following holds:
 
 \begin{enumerate}[label=(\roman*)]
     \item $M_1$ can be obtained from $M_2$ via complementation inside a cocircuit of $P_r$;
     
     \item $M_1$ can be obtained from $M_2$ via  complementation inside a hyperplane of $P_r$; or
     
     \item $M_1$ can be obtained from $M_2$ via complementation inside $P_r$.
 \end{enumerate}

\end{corollary}

\section{Local complementation and switching with respect to row cocircuits}

In this section, let $M$ denote a restriction of $P_r$. For local complementation in $M$, we do the following in order:

\begin{enumerate}[label=(\roman*)]

    \item Fix a cocircuit $C^*$ of $P_r$.
    \item Find the intersection $D$ of $C^*$ with  $E(M)$.
    \item Find the projective closure ${\rm{cl}}_{P_r}(D)$ of $D$.
    \item Complement inside ${\rm{cl}}_{P_r}(D)- C^*$.
\end{enumerate}

Note that ${\rm{cl}}_{P_r}(D)- C^*$ is a flat of $P_r$ and {\color{black} is} therefore  a smaller projective geometry. Thus local complementation in a matroid takes the complement inside a smaller projective geometry {\color{black} just as,  in} a graph, local complementation takes the complement inside a smaller complete graph. 

Next, we show that, {\color{black} in performing a local complementation in $M$, we can focus initially on} any disjoint union of cocircuits of $M$.

\begin{lemma}
\label{localcomp1}
Let $D$ be a {\color{black} non-empty} disjoint union of cocircuits of $M$. Then there is a projective cocircuit $C^*$ such that $C^* \cap E(M) = D$.
\end{lemma}

\begin{proof}
Observe that $M = P_r \ba T$ for some $T \subseteq E(P_r)$. Let $D$ be the disjoint union of cocircuits $D_1^*, D_2^*, \dots, D_k^*$ of $M$. Note that, for all $1 \leq i \leq k$,  there is a projective cocircuit $C_i^*$ of $P_r$ such that $D_i^* = C_i^* - T$. Then $D = D_1^* \bigtriangleup D_2^* \bigtriangleup \dots \bigtriangleup D_k^* = (C_1^* - T) \bigtriangleup (C_2^* - T) \bigtriangleup \dots \bigtriangleup (C_k^*-T) = (C_1^* \bigtriangleup C_2^* \bigtriangleup \dots \bigtriangleup C_k^*) - T$. The result follows.
\end{proof}

{\color{black} In view of this lemma, every local complementation in $M$ can be achieved as follows: 

\begin{enumerate}[label=(\roman*)]

    \item Find a non-empty disjoint union $D$ of cocircuits in $M$.  
    \item Find ${\rm{cl}}_{P_r}(D)$.
    \item Find a   hyperplane $H$ of  ${\rm{cl}}_{P_r}(D)$ avoiding $D$.
    \item Complement inside $H$.
\end{enumerate}}

{\color{black} Switchings in the graph case are only done} with respect to the vertex bonds. {\color{black} However, by Lemma~\ref{switch1},} via those switchings, we can complement inside any arbitrary bond of the complete graph. {\color{black} In the matroid case, we have been allowing switchings relative to any projective cocircuit. Next, we note that, just as for graphs, we can restrict ourselves to doing switchings with respect to a small number of cocircuits.}

Fix a standard representation $[I_r|D]$ of $P_r$ with respect to the basis $B$. Row $i$ of this matrix is the incidence vector of a cocircuit of $P_r$, a fundamental cocircuit with respect to the cobasis $E(P_r)-B$. We call such a cocircuit a \textbf{row-cocircuit}. The following lemma is well-known (see, for example, \cite{ox1} Proposition 9.2.2).


\begin{lemma}
\label{rowcocircuit1}
Every cocircuit $C$ of $P_r$ can be written as a symmetric difference of row-cocircuits.
\end{lemma}


This lemma has the following immediate consequence.

\begin{proposition}
\label{rowcocircuit2}
Switching with respect to an arbitrary cocircuit $C^*$ of $P_r$ can be expressed in terms of switchings with respect to row-cocircuits of $P_r$.
\end{proposition}

\section{Not all binary matroids are obtainable}

In this section, we show that, unlike for graphs, we cannot obtain all binary matroids of rank at most $r$ from $P_r$ using the three generalized operations.

The following result shows that the effect on a matroid $M$ of a local complementation with respect to a projective cocircuit $C^*$ is the same as performing complementation in the complementary projective hyperplane $E(P_r) - C^*$ if $E(M) \cap C^*$ has rank $r$.

\begin{lemma}
\label{lc1}
Let $C^*$ be a cocircuit of $P_r$ and $M$ be {\color{black} a} restriction of $P_r$ such that $E(M) \cap C^*$ has rank $r$. Then
$$\lambda_{C^*}(M)= \omega\sigma_{C^*}(M).$$
\end{lemma}

\begin{proof}
As $r(E(M) \cap C^*) = r$, we see that ${\rm{cl}}_{P_r}(E(M) \cap C^*) = E(P_r)$. Thus $\lambda_{C^*}(M) = M \triangle (E(P_r)-C^*)$. By Lemma \ref{hypcomp}, the last matroid is $\omega\sigma_{C^*}(M)$.

\end{proof}

The next result shows that the operation of complementation is redundant and can be expressed in terms of the operations of switching and local complementation. Recall that complementation and switching commute.

\begin{proposition}
\label{comp1}
Let $N$ be a matroid that can be obtained from $P_r$ using the operations of complementation, switching, and local complementation. Then the operations of switching and local complementation are enough to obtain $N$ from $P_r$.
\end{proposition}

\begin{proof}
It is enough to show that complementation of an arbitrary restriction $M$ of $P_r$ can be written in terms of switchings and local complementation. Suppose $r(M) < r$. Then there is a projective cocircuit $C^*$ of $P_r$ such that $C^* \cap E(M) = \emptyset$. 
{\color{black} Hence $C^* \cap E(\sigma_{C^*}(M)) = C^*$. 
As every cocircuit of $P_r$ is spanning,  by Lemma \ref{lc1}, we have $\lambda_{C^*}(\sigma_{C^*}(M)) = \omega\sigma_{C^*}(\sigma_{C^*}(M))$. Thus
 $\lambda_{C^*}(\sigma_{C^*}(M)) =  \omega(M)$.} 
We may now assume that $r(M) = r$. Then $M$ contains a basis $B$ of $P_r$. Since every projective cocircuit contains a basis of $P_r$, by the symmetry of $P_r$, the basis $B$ is contained in a projective cocircuit, say $C^*$. {\color{black} Then, using Lemma \ref{lc1} again, we get} $\omega(M) = \sigma_{C^*}\lambda_{C^*}(M)$.
\end{proof}

The following is the main result of this section.

\begin{theorem}
\label{main}
For fixed $r$ exceeding {\color{black} seven}, not all binary matroids of rank at most $r$ can be obtained from $P_r$ using the operations of complementation, local complementation, and switching.
\end{theorem}

\begin{proof}
Consider a restriction $M$ of $P_r$. We say that $M$ has \textbf{Property 1} if, for every two distinct  cocircuits $C^*$ and $D^*$ of $P_r$, both $(C^* - D^*) \cap E(M)$ and $(C^* - D^*)-E(M)$
have rank $r-1$. Instead, if, for every cocircuit $C^*$ of $P_r$, both $C^* \cap E(M)$ and $C^* - E(M)$ have rank $r$, then we say that $M$ has \textbf{Property 2}.

\begin{sublemma}
\label{sublemma1}
If $M$ has Property 1, then it has Property 2.
\end{sublemma}

Let $C^*$ be an arbitrary cocircuit of $P_r$. Let $e \in C^* \cap E(M)$. Observe that there is a different projective cocircuit $D^*$ such that $e \in C^* \cap D^*$. Since $M$ has Property 1, $r((C^*-D^*) \cap E(M)) = r-1$. 
As $e$ is not in the projective closure of $C^*-D^*$, we deduce that $r((C^*-D^*) \cap E(M)) < r(((C^*-D^*)\cap E(M)) \cup e)$, so $r(C^* \cap E(M)) = r$. Similarly, $r(C^* - E(M))=r$ unless $E(M) \supseteq  C^* \cap D^*$. In the exceptional case, {\color{black} since, by Lemma~\ref{symm2}, $C^* \triangle D^*$ is a cocircuit of $P_r$, and  $C^* \cap D^* = C^* - (C^* \bigtriangleup D^*)$, 
we get a violation of Property 1.} Hence \ref{sublemma1} holds.

We will use probabilistic methods for the rest of the proof. Independently colour each element of $P_r$ green or red with equal probability. Let $G$ and $R$ denote the sets of green and red elements, respectively. This gives $2^{2^r-1}$ members of a sample space,  $\mathbb{S}$, of all possible 2-colourings of $P_r$.

\begin{sublemma}
\label{sublemma2}
There is a $2$-colouring, say $X$, of $P_r$ such that, for all distinct cocircuits $C^*$ and $D^*$ of $P_r$, both $(C^*-D^*) \cap G$ and $(C^*-D^*) \cap R$ have rank $r-1$.
\end{sublemma}

{\color{black} Call a 2-colouring of $P_r$ bad if, for some pair of distinct cocircuits, $C^*$ and $D^*$, of $P_r$, either $(C^*-D^*) \cap G$ or $(C^*-D^*) \cap R$ has rank less than $r-1$. Note that $P_r|(C^*-D^*)$ is isomorphic to $A_{r-1}$.  Moreover, $(C^*-D^*) \cap G$ or $(C^*-D^*) \cap R$ has rank less than $r-1$ if and only if it is contained in an $A_{r-2}$. Thus a 2-colouring is bad if and only if there is a monochromatic $A_{r-2}$. 
Therefore, the number of bad colourings   is at most twice the product of the number of copies of  $A_{r-2}$ in   $P_r$ and the number of subsets of $E(P_r) - E(A_{r-2})$. The number of copies of  $A_{r-2}$ in  $P_r$ is the product of the number of copies of $P_{r-2}$ in $P_r$ and the number of hyperplanes of $P_{r-2}$. Thus the number of bad colourings is at most 
$$\frac{2(2^r-1)(2^{r-1} - 1)(2^{r-2} - 1)(2^{2^r - 1 - 2^{r-3}})}{3}.$$ 
Thus the probability that $P_r$ has a bad colouring is at most 
$$\frac{2(2^r-1)(2^{r-1} - 1)(2^{r-2} - 1)}{3(2^{2^{r-3}})}.$$}

\noindent For $r > {\color{black} 7}$, this probability is less than $1$, so \ref{sublemma2} holds.

Now, for the 2-colouring $X$, let $P_r|G = M$. Then, by \ref{sublemma1} and \ref{sublemma2}, $M$ {\color{black} satisfies Properties 1 and 2.} Note that $\omega(M)$ satisfies Properties 1 and 2; and $\sigma_{C^*}(M)$ satisfies Property 2 for all projective cocircuits $C^*$. Moreover, since $M$ satisfies Property 2, by Lemma \ref{lc1}, $\lambda_{C^*}(M) = \omega \sigma_{C^*}(M)$ and so $\lambda_{C^*}(M)$ also satisfies Property 2. 
{\color{black} Now suppose that $\alpha$ is a sequence of $n$ operations, each a complementation, a switching, or a local complementation. 
We show next that 

\begin{sublemma}
\label{sublemmaalpha}
$\alpha(M)$ is  $M$, $\omega(M)$, $\sigma_{C^*}(M)$, or $\sigma_{C^*} \omega(M)$, for some projective cocircuit $C^*$, and so $\alpha(M)$ satisfies Property 2.
\end{sublemma}

We obtain this from Lemma~\ref{switch and complemetation 0} by showing, by induction on $n$, that  $\alpha$ can be written as a  sequence of switchings and complementations. This was noted above for $n = 1$. Assume  it holds for $n< k$ and let $n = k \ge 2$. Then $\alpha = \gamma\beta$ 
where $\beta$ is the product of the first $k-1$ operations and $\gamma$ is the $k^{\rm th}$ operation. By the induction assumption, $\beta(M)$ has the specified form and satisfies Property 2. Thus, so does $\alpha(M)$ unless $\gamma = \lambda_{D^*}$. But, in the exceptional case, 
by Lemma~\ref{lc1}, $\alpha(M) = \lambda_{D^*}\beta(M) = \omega\sigma_{D^*}\beta(M)$ and the assertion again follows by Lemma~\ref{switch and complemetation 0}. We conclude that \ref{sublemmaalpha} holds.}

It follows that any matroid that can be obtained from $M$ by a sequence of the operations of complementation, switching, and local complementation satisfies Property 2. Since each of these operations has order two and $P_r$ cannot be obtained from $M$, we see that $M$ cannot be obtained from $P_r$.
\end{proof}

{\color{black} It is not difficult to check that, for $r$ in $\{1,2,3\}$, 
  all binary matroids of rank at most $r$ can be obtained from $P_r$ using the given operations. 
The next proposition establishes that this is also true when $r$ is $4$. By Theorem~\ref{main}, the corresponding result fails when $r$ is at least $8$. 
We do not know what happens when $r \in\{5,6,7\}$.}

\begin{proposition}
\label{bound}
All binary matroids of rank at most four can be obtained from $P_4$ using the operations of complementation, switching, and local complementation.
\end{proposition}

\begin{proof}
Because we can use complementation inside of projective hyperplanes, to see that every matroid of rank at most three can be obtained, it suffices to show that a member of each of the following pairs can be obtained from $P_4$: $\{P_3, U_{0,0}\}$, $\{M(K_4), U_{1,1}\}$, $\{P(U_{2,3}, U_{2,3}), U_{2,2}\}$ $\{U_{3,4}, U_{2,3}\}$, $\{U_{2,3} \oplus U_{1,1}, U_{3,3}\}$. Now $\sigma_{C^*}(P_4)=P_3$. A local complementation in $P_3$ gives $U_{3,4}$. Doing a local complementation using a 2-cocircuit of $U_{3,4}$ gives {\color{black} $M(K_4 \backslash e)$.} In this matroid, there is 3-cocircuit and a 4-element set that is a disjoint union of two 2-cocircuits. Local complementation with respect to the first of these sets gives $U_{3,3}$; local complementation with respect to the second gives $M(K_4)$. We conclude that every matroid of rank at most three is obtainable from $P_4$. 

Assume the result fails and let $M$ be a minimum-sized matroid that cannot be obtained from $P_4$. Observe that $r(M)=4$, $|E(M)| \leq 7$, and that every hyperplane of $M$ has exactly three elements. First, we show that $M$ is not isomorphic to $U_{4,4}$. Observe that there is a projective cocircuit $C^*$ such that $|C^* \cap E(U_{4,4})| =3$. Performing a local complementation with respect to $C^*$ transforms $U_{4,4}$ to $U_{1,1} \oplus M(K_4)$. Now, complementation in the hyperplane containing the ground set of $M(K_4)$ gives us $U_{2,2}$. Since $U_{2,2}$ can be obtained from $P_4$, so too can $U_{4,4}$.

Observe that if $M$ has a coloop, then $M \cong U_{4,4}$. Assume all cocircuits of $M$ have at least three elements. Then $M^*$ is simple of corank four and rank at most three. Thus $M \cong F_7^*$. Note that $\omega(F_7^*) = F_7 \oplus U_{1,1}$. Now, complementation in the hyperplane $F_7$ gives us $U_{1,1}$ and therefore, $F_7^*$ can be obtained from $P_4$ using the given operations.

 We may now assume that $M$ has a 2-element cocircuit. Since $M$ has no coloops, $M \cong U_{4,5}$. Using local complementation with respect to a 2-cocircuit, we can transform $U_{4,5}$ into $P(U_{3,4}, U_{2,3})$. Now, doing a complementation in the hyperplane containing the ground set of $U_{3,4}$ gives us $U_{2,3} \oplus U_{2,2}$. Finally, performing a local complementation using a 2-cocircuit of $U_{2,3} \oplus U_{2,2}$, we obtain $U_{4,4}$ implying that $U_{4,5}$ is obtainable via the given operations, a contradiction.   \end{proof}

\section{More Operations}

In this section, we introduce a new operation. Let $M$ be a restriction of $P_r$. Colour an element $e$ of $P_r$ green if $e \in E(M)$ and red otherwise. Denote the sets of green and red elements of $P_r$ by $G$ and $R$,  respectively. Note that $E(M) = G$. Let $f$ be an element of $P_r$. For every line $L = \{f,f',f''\}$ passing through $f$, we consider $\{f',f''\}$, and swap their colours if these {\color{black} colours are different.} We call this operation a \textbf{pointed swap with respect to} $f$. It is an \textbf{on-element swap} if $f\in G$ and an \textbf{off-element swap} otherwise. We use $\psi^+_{f}(M)$ and $\psi^-_{f}(M)$ to denote the matroids obtained from $M$ by doing on-element and off-element swaps with respect to $f$.

This new operation can also be viewed in terms of matrices. Let $A$ be an $r \times t$ matrix over $GF(2)$ representing $M$ such that the columns of $A$ correspond to the green elements of $P_r$. Let $v$ be a binary vector of length $r$ corresponding to an element $f$ of $P_r$. By $A \hat{+}v$, we denote the matrix obtained by adding $v$ to every column of $A$ that is distinct from $v$. Observe that $M[A\hat{+}v]$ is the matroid obtained from $M[A]$ by a pointed swap with respect to $f$. This operation is an on-element swap if $f\in G$ and an off-element swap otherwise.

\begin{lemma}
\label{on-off1}
Let $M$ be a $t$-element matroid that is a restriction of $P_r$. Then every $t$-element restriction of $P_r$ can be obtained from $M$ using pointed swaps.
\end{lemma}

\begin{proof}
Let $A= [v_1, v_2, \ldots, v_t]$ be a matrix representing $M$, so the columns of $A$ correspond to the green elements of $P_r$. Note that it is enough to show that, for any vector $w$ corresponding to a red element $e$ of $P_r$ and an arbitrary green element $f$ of $P_r$ corresponding to the column $v_k$, we can use pointed swaps to obtain a matroid $M'$ from $M$ such that $M'= P_r|(E(M) \bigtriangleup \{e,f\})$. Doing an off-element swap on $M$ with respect to $w$, we get

\begin{align*}
    [v_1,\ldots,v_k, \ldots, v_r] \xrightarrow{\psi^-_w}  [v_1 +w, \ldots, v_k+w, \ldots, v_r+w].
\end{align*}

\noindent Now, doing an on-element swap with respect to $v_k+w$, we get

\begin{align*}
[v_1 +w, \ldots, v_r+w] \xrightarrow{\psi^+_{v_k+w}}  
    [v_1+v_k, \ldots, v_k+w, \ldots, v_r+v_k]. 
\end{align*}

\noindent Finally, doing an off-element swap with respect to $v_k$, we get

\begin{align*}
    [v_1+v_k, \ldots, v_k+w, \ldots, v_r+v_k] \xrightarrow{\psi^-_{v_k}} [v_1,\ldots, w, \ldots, v_r].
\end{align*}

\end{proof}

Next we show that the operation of complementation can be obtained by three complementations inside of projective hyperplanes.

\begin{lemma}
\label{hyperplane_complement}

For $r>1$, let $H_1, H_2, H_3$ be three projective hyperplanes that contain a fixed rank-$(r-2)$ flat of $P_r$. Then $\omega(M) = M \triangle H_1 \triangle H_2 \triangle H_3$.
\end{lemma}

\begin{proof}
By Lemma \ref{symdiff}, $M \triangle H_1 \triangle H_2 \triangle H_3 = M \triangle (H_1 \bigtriangleup H_2 \bigtriangleup H_3)$. But $H_1 \bigtriangleup H_2 \bigtriangleup H_3 = E(P_r)$, so $M \triangle H_1 \triangle H_2 \triangle H_3 = \omega(M)$.
\end{proof}

\begin{theorem}
\label{main1}
 For $r > 1$, all binary matroids of rank at most $r$ can be obtained from $P_r$ via a sequence of the operations of pointed swaps and complementation inside projective hyperplanes.

\end{theorem}

\begin{proof}
Assume the theorem fails and let $M$ be a matroid with minimum-sized ground set that cannot be obtained from $P_r$ via the specified operations. First, we show that $U_{1,1}$ can be obtained from $P_r$, and so $M$ has at least two elements. Doing a complementation inside a hyperplane of $P_r$, we obtain a matroid $N$ that is isomorphic to $A_r$ and thus has $2^{r-1}$ elements. As ${\color{black} P_{r-1}} \oplus U_{1,1}$ also has $2^{r-1}$ elements, Lemma \ref{on-off1} implies we can get a matroid isomorphic to $P_{r-1} \oplus U_{1,1}$ from $N$ via pointed swaps. Now, complementation inside the hyperplane $P_{r-1}$ gives us $U_{1,1}$. 

Let $x, y$ be two distinct elements in the ground set of $M$ and let $C^*$ be a projective cocircuit containing $x$ and $y$. Let $H = E(P_r)-C^*$. Recall that $x$ and $y$ are coloured green in $P_r$. We may assume that $H$ has at least one red element, say $z$, otherwise, by complementation inside of $H$, we reduce the number of green elements. Note that {\color{black} the colours of any pair of differently coloured elements of $P_r$ can be interchanged by a single pointed swap.} Thus we can swap the colours of $x$ and $z$ to get a matroid $M'$ such that $E(M')= E(M) \bigtriangleup \{x,y\}$. 
{\color{black} Lemma~\ref{hyperplane_complement} allows us to perform complementation on $M'$ followed by complementation inside the hyperplane $H$.} 
Observe that, after this operation, $y$ becomes a red element and $x$ becomes a green element. Again, using pointed swaps, we swap the colours of $y$ and $z$. Finally, doing switching with respect to $C^*$ gives us a matroid $M''$ such that $E(M'') = E(M)- \{x,y\}$. Since $|E(M'')| < |E(M)|$, $M''$ is obtainable from $P_r$ using the allowed operations and so, $M$ is obtainable as well, a contradiction.
\end{proof}


We now show that the operations of on-element and off-element swaps are linked via complementation.

\begin{lemma} 
\label{on_offswap1}
Let $M$ be a restriction of $P_r$ and $u$ be an element of $P_r$. Then 

$$\psi^+_{u}(M) = \omega \psi^-_{u}\omega(M) {\color{black} \text{~~ when $u \in E(M)$,}} $$
and 
$$\psi^-_{u}(M) = \omega \psi^+_{u}\omega(M) {\color{black} \text{~~ when $u \not\in E(M)$.}}$$

\end{lemma}

\begin{proof}
To show the first part, it suffices to show that $\omega \psi^+_{u}(M) =  \psi^-_{u}\omega(M)$. In considering $\psi^+_{u}(M)$, we note that $u$ must be green. The following statements are equivalent for an element $e$ of $P_r$:

\begin{enumerate}[label=(\roman*)]
    \item $e$ is green in $\omega \psi^+_{u}(M)$;
    
    \item $e$ is red in $\psi^+_{u}(M)$;
    
    \item $e+u$ is red in $M$;
    
    \item $e+u$ is green in $\omega(M)$;
    
    \item $e$ is green in $\psi^-_{u}\omega(M)$.
\end{enumerate}

\noindent Thus the first part holds. 

For the second part, let $u$ be red in $P_r$. Then $u$ is green in $\omega(M)$. Thus, by the first part,
$\psi^+_{u}(\omega(M)) =  \omega \psi^-_{u}\omega(\omega(M))$, so $\psi^+_{u}\omega(M) =  \omega \psi^-_{u}(M)$. Hence $ \omega\psi^+_{u}\omega(M) = \psi^-_{u}(M)$ and the second part holds.

\end{proof}

By combining the last three results we obtain the following.

\begin{theorem}
\label{main2}
Let $P_r$ be a binary projective geometry of rank $r$ greater than one. Then all binary matroids of rank at most $r$ can be obtained from $P_r$ via a sequence of the operations of pointed swaps and  complementation in hyperplanes, where either all such swaps are on-element swaps, or they are all off-element swaps.
\end{theorem}

\begin{proof}
By the last lemma, when one has the operation of complementation, one needs only on-element or off-element swaps but not both. By Lemma \ref{hyperplane_complement}, complementation can be achieved via a sequence of complementations inside of projective hyperplanes. The theorem follows from Theorem \ref{main1}.
\end{proof}

\begin{theorem}
\label{main3}
Let $P_r$ be a binary projective geometry of rank $r$ and $M$ be a binary matroid of rank at most $r$ such that $M$ is not isomorphic to $U_{0,0}$ or $U_{1,1}$. Then $M$ can be obtained from $P_r$ via a sequence of the operations of local complementation and pointed swaps.

\end{theorem}

Our proof of this theorem will use the following two lemmas.

\begin{lemma}
\label{lc11}
Let $M$ be a $k$-element restriction of $P_r$ such that $e$ and $f$ are coloops of $M$. Then there is a $(k+1)$-element matroid $M'$  that can be obtained from $M$ via local complementation.
\end{lemma}

\begin{proof}
Let $L=\{e,f,g\}$ be the line containing $e$ and $f$ in $P_r$. By Lemma ~\ref{localcomp1}, there is a projective cocircuit $C^*$ such that $C^* \cap E(M) = \{e,f\}$. Since $|C^* \cap L|$ cannot be odd, $g \notin C^*$. Note that $g \notin E(M)$ and $E(\lambda_{C^*}(M)) = E(M) \cup g$. Therefore, the result holds.
\end{proof}

\begin{lemma}
\label{lc22}
Let $r \geq 3$. For an integer $k$ in  $[2, 2^{r-2}+1]$, there is a restriction $M$ of $P_r$ such that $|E(M)|=k$ and $M$ has two coloops. 
\end{lemma}

\begin{proof}
Let $F$ be a flat of $P_r$ of rank $r-2$ and $M$ be a restriction of $P_r$ such that $|E(M)| = |E(M|F)| + 2$ and $r(M) = r(M|F) + 2$. Note that $M$ has two coloops. The maximum number of elements $M$ can have is $|F| + 2$, that is, $2^{r-2}+1$ . Therefore, our result follows.
\end{proof}

\begin{proof}[Proof of Theorem \ref{main3}]
We can check that the result is true for $r=1,2$, and $3$. Therefore, we may assume that $r \geq 4$. We say $M_1 \sim M_2$ if $M_1$ is obtainable from $M_2$ via the given operations. Since both the operations have order two, the relation $\sim$ is symmetric and so is an equivalence relation.

First, we show that all binary matroids with at least two and at most $2^{r-2} + 2$ elements can be obtained from $U_{2,2}$ via the operations above. Given a matroid $M$ with at least two coloops and at most $2^{r-2}+1$ elements altogether, Lemma {\color{black}\ref{lc11}} implies that we can obtain a matroid with $|E(M)|+1$ elements. Thus, for all integers $k$ in $[2, 2^{r-2}+2]$, we can construct a $k$-element matroid starting with $U_{2,2}$. Hence, by {\color{black} Lemma~\ref{on-off1},} using pointed swaps, we can construct every $k$-element matroid from $U_{2,2}$.

Let $M$ be a matroid with $2^{r-2}+2$ elements. Via pointed swaps, we can obtain a rank-$r$ matroid $M'$ from $M$ such that $E(M') \subseteq C^*$ for a projective cocircuit $C^*$. Note that $ \lambda_{C^*}(M')$ has size $2^{r-2}+2^{r-1}+1$ and is obtainable from $U_{2,2}$.

\setcounter{theorem}{6}

We now show the following.

\begin{sublemma}
\label{newsublemma1}
For $0 \leq k \leq 2^{r-1}-r$, every binary matroid with $(2^r-1)-2k$ elements can be obtained from $P_r$ using the given operations.
\end{sublemma}

 Recall that an element $e$ of $P_r$ is coloured green if $e$ is in the ground set of the present matroid and red otherwise. We start with $P_r$ and do a local complementation with respect to a projective cocircuit $C^*$ to obtain a matroid $N$ whose ground set is $C^*$. Let $B$ be a basis of $P_r$ contained in $C^*$. Pick $k$-element subsets of each of $C^*-B$ and $E(P_r)-C^*$ and, {\color{black} by Lemma~\ref{on-off1},  pointed swaps enable us to} interchange the colours on these $2k$ elements to get a matroid $N'$. Now, we do a local complementation in $N'$ with respect to $C^*$ to obtain a matroid $N''$ with $ (2^r - 1) - 2k$ elements. Thus \ref{newsublemma1} holds.

 The above implies that all matroids of odd size between $2^{r}-1$ and $2r-1$ can be obtained from $P_r$ using the allowed operations. Therefore, the $(2^{r-2}+2^{r-1}+1)$-element matroid $\lambda_{C^*}(M')$, which was constructed above from $U_{2,2}$, can also be constructed from $P_r$. Since $\sim$ is an equivalence relation, from $P_r$, we can obtain all matroids with $t$ elements for all $t$ in $[2,2^{r-2}+2]$. This implies that all matroids other than $U_{1,1}$ that have an odd number of elements are obtainable from $P_r$.

Finally, we show that all matroids with a non-zero even number of elements are obtainable from $P_r$. Since $|C^*| - 1$ is odd for a projective cocircuit $C^*$, we can obtain a matroid $N$ from $P_r$ such that $E(N) = C^* - \{e\}$. Let $k$ be an integer in $[0, 2^{r-1}-r-1]$. Let $B$ be a basis of $P_r$ contained in $C^* - \{e\}$. Pick $k$-element subsets of each of $C^*-(B \cup e)$ and $E(P_r)-C^*$ and, using pointed swaps, interchange the colours on these $2k$ elements to get a matroid $N'$. Now, we do a local complementation in $N'$ with respect to $C^*$ to obtain a matroid $N''$ of size $ (2^r - 2) - 2k$. Observe that, when $k=2^{r-1}-r-1$, the matroid $N''$ has $2r$ elements. Since $r \geq 4$ and all matroids having at most $2^{r-2} + 2$ elements are obtainable from $P_r$, the matroids having even size less than $2r$ are also obtainable from $P_r$. This completes the proof. \end{proof}

{\color{black}\section*{Acknowledgement} 
The authors thank the referee for carefully reading the paper and, in particular, for detecting an error in the proof of Theorem~\ref{main}.}

\end{document}